\documentclass{article} 

\usepackage{amsmath,amsthm}     
\usepackage{graphicx}     
\usepackage{hyperref} 
\usepackage{url}
\usepackage{amsfonts} 
\usepackage{float}

\usepackage{multicol}

\newcommand{\E}{\mathbb{E}}
\renewcommand{\H}{\mathbb{H}}

\newcommand{\R}{\mathbb{R}}

\renewcommand{\line}[1]{\overleftrightarrow{#1}}
\renewcommand{\S}{\mathbb{S}}
\renewcommand{\H}{\mathbb{H}}

\DeclareMathOperator{\gsin}{gsin}

\theoremstyle{theorem}
\newtheorem{theorem}{Theorem}
\newtheorem{lemma}[theorem]{Lemma}
\newtheorem{corollary}[theorem]{Corollary}

\theoremstyle{definition}
\newtheorem*{definition}{Definition}
\newtheorem*{remark}{Remark}

\allowdisplaybreaks

\makeatletter
\@addtoreset{footnote}{page}
\makeatother

\begin{document}

\title{Non-Euclidean Cross-Ratios and Carnot's Theorem for Conics}

\author{Michael Perez Palapa\footnote{Independent Researcher}, Kai Williams\footnote{Swarthmore College}\\               %%%% Leave ALL of these as 
kawilliams@protonmail.com}                      %%%% your final manuscript.

\maketitle

Despite significant work done in non-Euclidean geometries from a synthetic perspective, we still know far more about planar Euclidean geometry ($\E^2$) than even the nicer non-Euclidean geometries: spherical ($\S^2$) and hyperbolic ($\H^2$). This presents an opportunity: given the deep connections between these three different geometries, can we use facts we know to be true in Euclidean geometry to obtain results in spherical and hyperbolic geometry? 

In some sense, this idea is the motivation for using models to study hyperbolic geometry (or more broadly, for the extensive use of local coordinates in differential geometry). However, we can take this idea further: \cite{ProjectiveShadowBook} and \cite{ProjectiveShadows}  use the idea of a ``projective shadow." Projective geometry, which has a rich history dating back to ancient times, is a type of geometry which contains Euclidean, spherical, and hyperbolic geometry. It concerns itself with properties invariant under projective transformations, such as conics and collinearity (for more, see \cite{Coxeter}). Because projective geometry is a more fundamental geometry, one can often take theorems from it to find analogous results -- ``shadows" -- in spherical and hyperbolic geometry. These shadows are often quite distinct from the original results and their proofs can require considersible creativity, but this is still a fruitful perspective from which to consider spherical and hyperbolic geometry.

In this paper, we find a particularly important shadow of a fundamental tool in projective geometry -- the cross-ratio -- and use it to extend many theorems true in Euclidean and projective geometry directly into spherical and hyperbolic geometry. After discussing preliminary background, we will develop a notion of the cross-ratio. Alongside the use of central projection (or alternatively, the Beltrami-Klein model), this will allow us to generalize several theorems in section 2 and section 3, including a classic result: Carnot's Theorem for Conics.

\begin{center}
    \begin{figure}[H]
        \centering
        \includegraphics[scale=0.8]{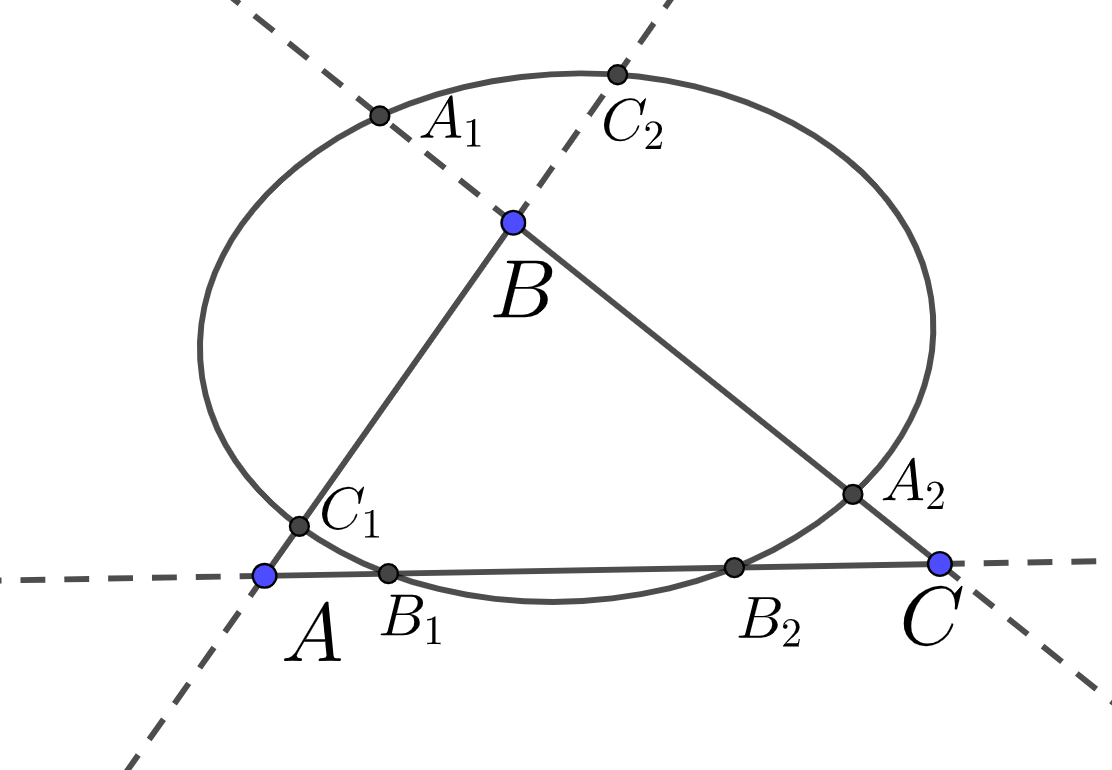}
        \caption{$A_1, A_2, B_1, B_2, C_1$, and $C_2$ lie on a conic if and only if the above equation holds.}
        \label{fig:CarnotStatement}
    \end{figure}
\end{center}
\begin{theorem}[Carnot's Theorem for Conic]
Let $\triangle ABC$ be an arbitrary triangle in $\E^2$. Let $A_1,A_2$, $B_1,B_2$, and $C_1,C_2$ be points on the sides $\line{BC}$, $\line{AC}$, and $\line{AB}$ respectively. Then, $A_1$, $A_2$, $B_1$, $B_2$, $C_1$, and $C_2$ lie on a conic if and only if 
\begin{equation} \label{EuclidCarnot}
\frac{{AC_1}}{{C_1B}}\cdot\frac{{AC_2}}{{C_2B}}\cdot\frac{{BA_1}}{{A_1C}}\cdot\frac{{BA_2}}{{A_2C}}\cdot\frac{{CB_1}}{{B_1A}}\cdot\frac{{CB_2}}{{B_2A}} = 1
\end{equation}
where we measure the signed length of each segment.
\end{theorem}

\section{Non-Euclidean Geometry and other preliminaries}
To work in non-Euclidean geometry, we will choose appropriate models for spherical ($\S^2$) and hyperbolic ($\H^2$) geometry and introduce the key terminology.

{\it Spherical Geometry} ($\S^2$) We will model spherical geometry on the unit sphere $x^2+y^2+z^2=1$. \emph{Geodesics} (sometimes called \emph{great circles}) are the intersection of the sphere with a plane passing through the origin (see figure \ref{fig:SphereLine}). For almost every pair of points, there is a unique geodesic passing through this pair of points; however, any two points formed by the intersection of the sphere with a line passing through the origin are \emph{antipodal}, and there are an infinite number of geodesics passing through such a pair of points.

 \begin{figure} \begin{center} \includegraphics[scale=0.4]{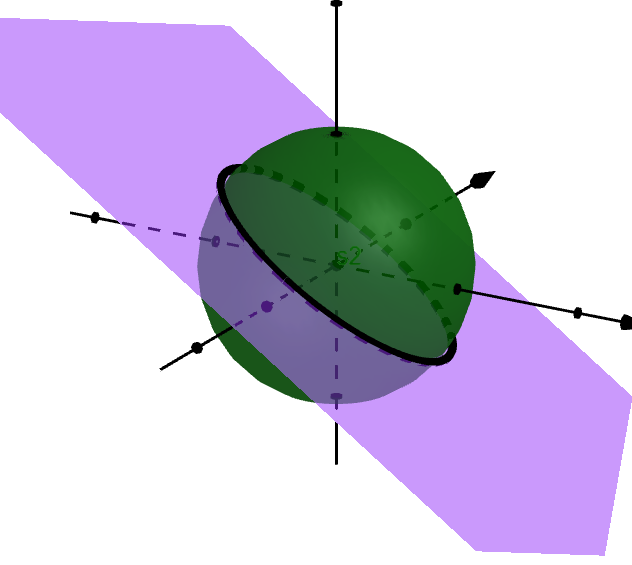} \caption{A geodesic in $\S^2$} \label{fig:SphereLine} \end{center} \end{figure} 

There are two line segments connecting any two non-antipodal points, whose total length sums to $\pi$. We say that for any three non-collinear points $A$, $B$, and $C$, no two of which are antipodes, the triangle $\triangle ABC$ is composed of the shorter segments between each pair of points. We include this restriction to preserve basic facts about the triangle, such as the triangle inequality. We measure angle between two geodesics as the angle in $\R^3$ between their tangent vectors in the tangent plane. We meausre distance using the induced metric of $\R^3$; i.e. $ds^2=dx^2+dz^2+dz^2.$ For more details, see \cite{SphericalTextbook}

{\it Hyperbolic Geometry} ($\H^2$): We will choose the \emph{hyperboloid model}, set on the upper sheet of the two sheet hyperboloid $x^2+y^2-z^2=-1$, $z>0$, in Minkowski 3-space $\mathbb{M}^3$. Geodesics are the intersection of the hyperboloid with a plane passing through the origin. We measure distance and angles using the \emph{Minkowski metric}: $ds^2=dx^2+dy^2-dz^2$. For more details, see \cite{Reynolds}

\begin{figure} \begin{center} \includegraphics[scale=0.2]{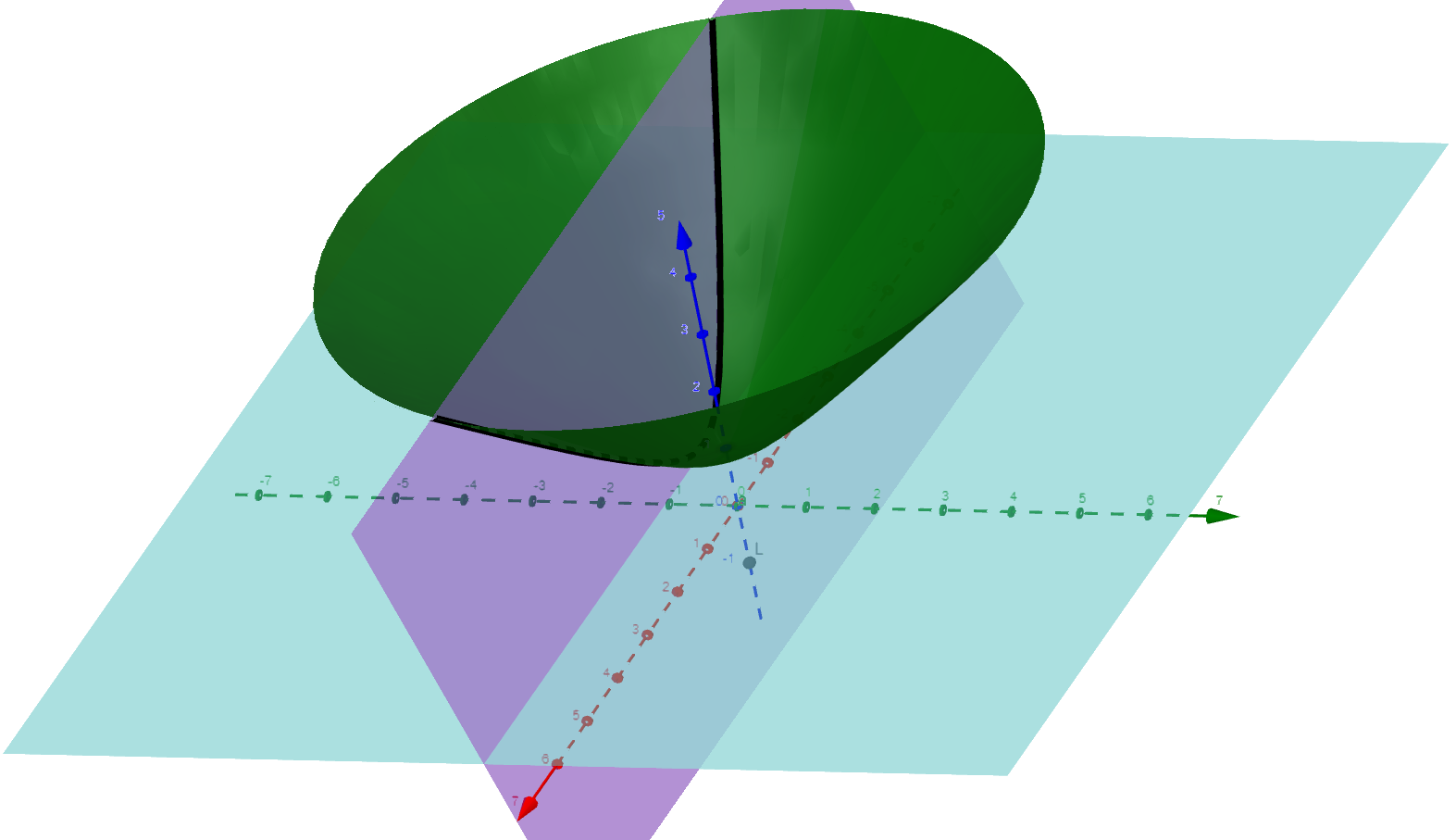} \caption{A geodesic in $\H^2$} \end{center} \end{figure}

With these definitions in mind, one can prove that the law of sines generalizes into $\S^2$ and $\H^2$. For any triangle $\triangle ABC$ in $\S^2$ or $\H^2$ respectively,
\[\frac{\sin A}{\sin a} = \frac{\sin B}{\sin b} = \frac{\sin C}{\sin c} \text{ in }\S^2 \qquad \text{ or } \qquad \frac{\sin A}{\sinh a} = \frac{\sin B}{\sinh b} = \frac{\sin C}{\sinh c}\text{ in }\H^2. \]
We can compress these formulas into one, using the \emph{generalized sine function}: 
\[\gsin(x)=x-\frac{Kx^3}{3!}+\frac{K^2x^5}{5!}-\frac{K^3x^7}{7!}+\ldots \]
where $K$ is a parameter we will regard as the Gaussian curvature $K$ of the constant-curvature surface. When $K=1$ (corresponding to $\S^2$), $\gsin(x)=\sin(x)$, when $K=0$ (corresponding to $\E^2$), $\gsin(x)=x$, and when $K=-1$ (corresponding to $\H^2$), $\gsin(x)=-1$, so we can write the law of sines in all geometries as 
\begin{equation}\label{LawOfSines}
\frac{\sin A}{\gsin a}=\frac{\sin B}{\gsin B}=\frac{\sin C}{\gsin C}.
\end{equation}
We will use this notation extensively.

We can also extend a notion of a conic:
\begin{definition} A \emph{conic} in $\S^2$ ($\H^2$) is the intersection of a quadratic cone (e.g. a cone whose boundary curve is any conic) centered at the origin with the sphere (hyperboloid).
\end{definition}
\begin{figure} \begin{center} \includegraphics[scale=0.4]{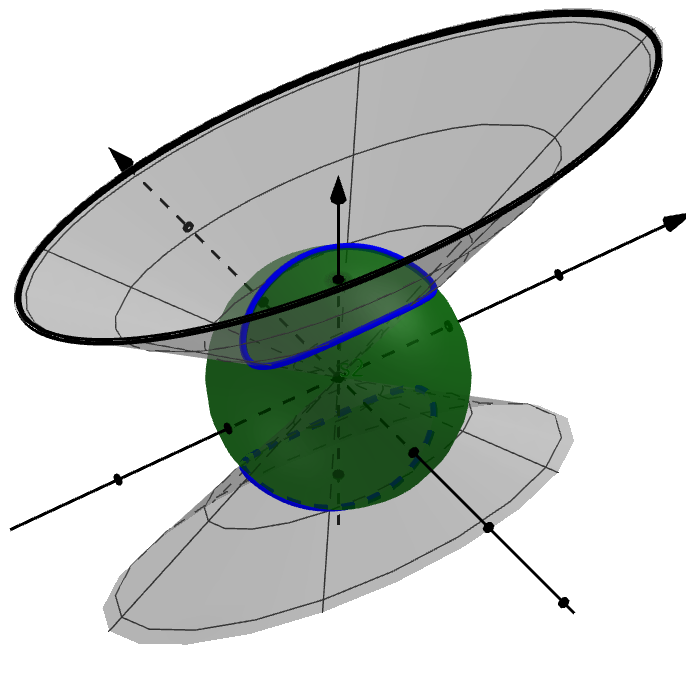} \caption{An ellipse in $\S^2$} \label{fig:SphereEllipse} \end{center} \end{figure} 

We see a picture of a spherical ellipse in Figure \ref{fig:SphereEllipse}. The quadratic cone in question is a cone whose base is an ellipse, as we see. Here, we assume that if a point is on a spherical conic, so is its antipode. While this definition initially seems to be distinct from the Euclidean definitions using two foci, \cite{ConicProjection} proves that the definition is equivalent (up to antipodes). Thus, the spherical ellipse pictured in Figure \ref{fig:SphereEllipse} has two foci where the sum of the distances to these foci is constant. This definition is also equivalent to a version of the focus-directrix definition  (\cite{SphericalandHyperbolicConics}). \\ 

{\it Central Projection}
One crucial tool we will use is projection from one geometry to another. While stereographic projection between the sphere and the plane is well studied (see \cite{IntrinsicCrossRatio} for ideas similar to this paper, but using stereographic projection), we will use \emph{central} (elsewhere referred to as \emph{gnomic}) projection repeatedly.

Let $w$ be a plane in $\R^3$ not containing the origin. Then, let $\pi:\S^2(\H^2)\to w$ be the central projection from the origin from the unit sphere or (the upperupper sheet of the hyperboloid) onto the plane $w$. That is, for each point $P$ in some subset of the sphere the sphere (hyperboloid), draw a line from the origin through $P$ and extend it until it intersects $w$ at some point $P'$ if $P'$ exists. Then $\pi(P)=P'$.  (See Figure \ref{fig:SphereEuclidProjection} for an example).

There are several important details to note. Some points on the sphere (and depending on our plane $w$, the hyperboloid) will not intersect the plane $w$ when projected. One can interpret these points as being sent to a line at infinity (so as to make $w$ a projective plane), but we will avoid this detail by restricting the domain of $\pi$ to points where the line intersects $w$.

More important is the relationship of projection to antipodal points on the sphere. Central projection sends any pair of antipodal points to the same image (by definition, they lie on a line passing through the origin). Thus, we will usually consider central projection from an open hemisphere of $\S^2$ to $w$. Should any point in question not lie in that open hemisphere, we can substitute its antipode and get the same projected image. We will use this trick several times later on.

When we project $\H^2$, we will generally set $w$ to be the plane $z=1$. From this, we obtain the Beltrami-Klein model, which has tight connections to projective geometry \cite{Reynolds}. If easier, one can think of many of the following proofs as being set in the Beltrami-Klein model.

Finally, central projection has several nice properties that we will use later on. As it is smooth and one-to-one (when restricted to an open hemisphere of $\S^2$, at least), it preserves incidence and tangency. Even nicer, it sends geodesic segments to geodesic segments. This follows from our geometric definition of a geodesic as the central projection of a plane through the origin is a line. 
Similarly, this projection sends a conic section to a conic section \cite{SphericalandHyperbolicConics}. This follows from the definition as well: the central projection of a quadratic cone is a quadratic curve: i.e., a conic. This is akin to how hyperbolic conics appear as Euclidean conics in the Beltrami-Klein model.

\section{Non-Euclidean Cross-Ratios}

As noted at the end of the previous section, central projection from the sphere (hyperboloid) to a plane preserves incidence, geodesics, and conics. This is also true of central projection between two planes. However, central projection between two planes (which corresponds to a projective transformation of the plane) preserves another important property: the \emph{cross-ratio}.

\begin{definition}
    Let $A$, $B$, $C$, and $D$ be four points on an oriented line $\ell$ in $\E^2$. (See Figure 5). Then, define the $\emph{cross-ratio}$ of these four points as
    \[(ABCD)_{\E^2}=\frac{{AC}}{{AD}}\cdot\frac{{BD}}{{BC}},\]
    where we measure the signed Euclidean distance of each segment with respect to the orientation of $\ell$.
\end{definition}

\begin{figure} \begin{center} \includegraphics[width=\textwidth]{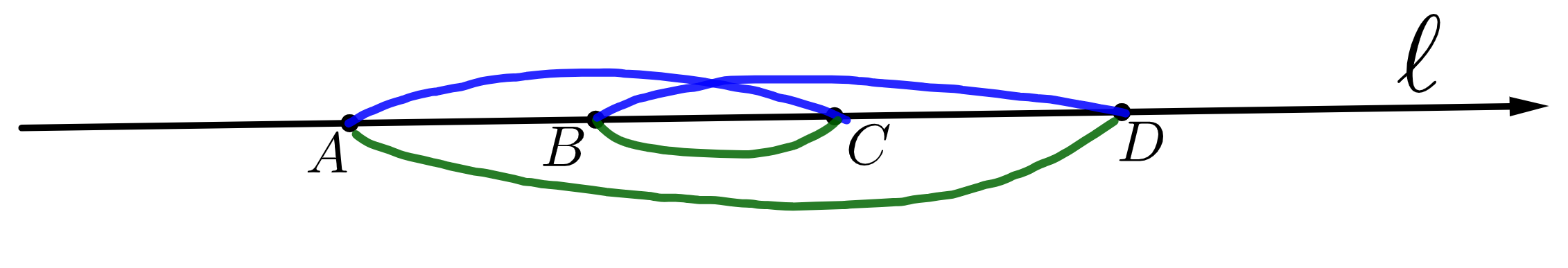} \caption{The Cross-Ratio on the geodesic $\ell$} \end{center} \end{figure} 

\begin{remark}
    If $A$, $B$, $C$, and $D$ lie in some different order along the line, then the orientation of the line becomes important as $AC=-CA$, for instance. However, if we flip the orientation of the line while keeping the points the same, the cross-ratio will be preserved. Thus, we will tend to ignore the specific orientations of the lines in question, except when it arises as a technical detail.
\end{remark}

We briefly note here an essential property of cross-ratios, which we will see later. This theorem follows from Lemma \ref{lemma:fundamentaltheorem} later on because $A$, $B$, $C$, $D$, $A',B',C',D'$ and $O$ all lie in a plane in $\R^3$.

\begin{theorem}\label{thm:EuclideanInvariance}
    If $A'$, $B'$, $C'$, $D'$ are the images of $A$, $B$, $C$, and $D$ under central projection from a point $O$ between two planes in $\R^3$, $(ABCD)=(A'B'C'D')$. 
\end{theorem}

We can extend this definition into non-Euclidean geometry:

\begin{definition}
    Let $A$, $B$, $C$, $D$ be four distinct pairwise non-antipodal points on an oriented geodesic $\ell$ in $\S^2$, $\E^2$, or $\H^2$. We say the (projective) cross-ratio of the four points is 
    \[(ABCD)=\frac{\gsin{AC}}{\gsin{AD}}\cdot\frac{\gsin{BD}}{\gsin{BC}}\]
    where $\gsin(AC)$ is the generalized sine of the oriented distance between $A$ and $C$ along the geodesic $\ell$.
\end{definition}

This extension of the cross-ratio into spherical geometry is ancient. It was first introduced by Menelaus in his second century work {\it Spherics} \cite{Spherics} and was further extended by several medieval Arab mathematicians \cite{AncientCrossRatio}. An equivalent definition was independently derived in an engineering context in \cite{EngineerCrossRatio}. 
A distinct extension of the cross-ratio into non-Euclidean geometry is considered in \cite{IntrinsicCrossRatio}, but it has tighter connections to stereographic projection and M\"{o}bius transformations, rather than projective transformations. Despite these preceding works, the main result of this section, Theorem \ref{CrossLineCrossPoints} is new, as well as the hyperbolic extension of these ideas.

\begin{figure} \centering \includegraphics[scale=0.5]{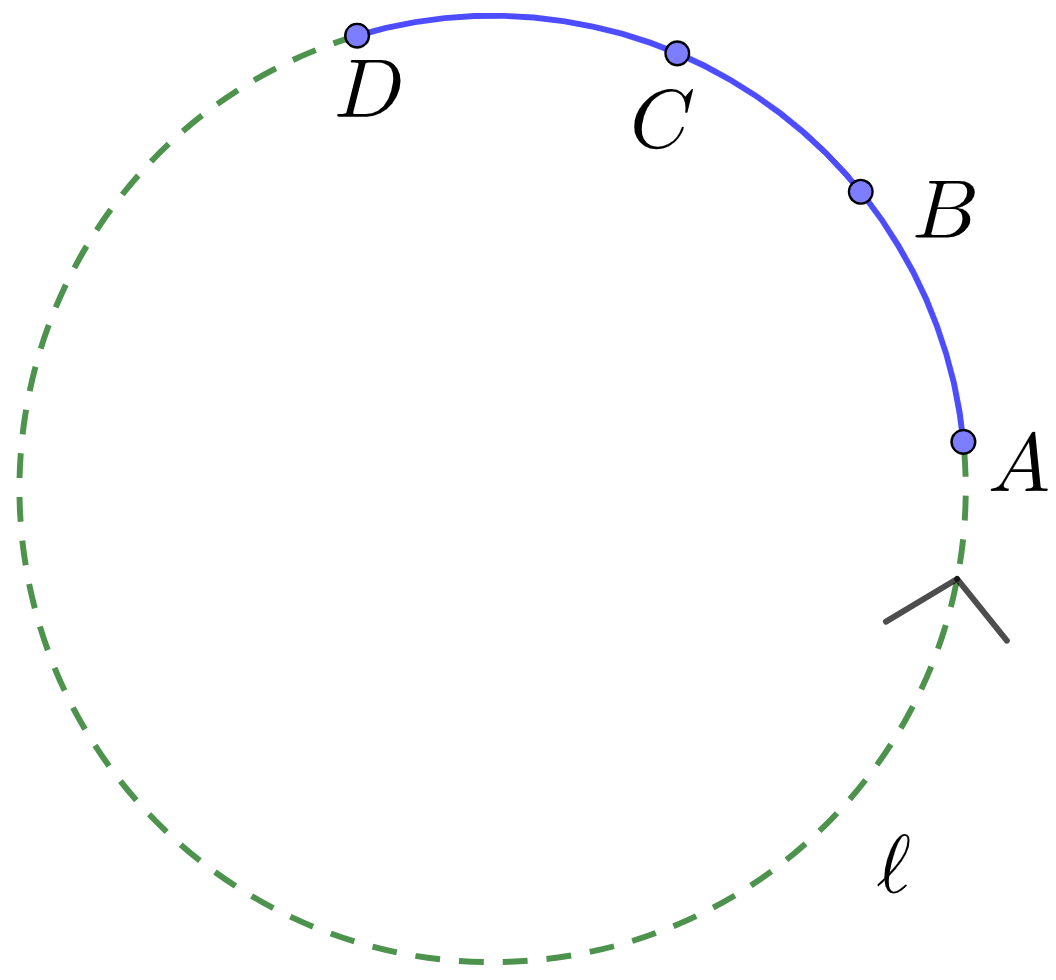} \caption{The Cross-Ratio along a great circle} \label{crossratioorientation} \end{figure} 

Before we use this definition, we must verify it is well defined in spherical geometry. As noted earlier, for any two non-antipodal points $A$ and $D$ on the sphere, there are two line segments between $A$ and $D$: one of length $d<\pi$, and the other of length $2\pi-d$. So we might worry that measuring the distance in the other direction would change the ratio 
(e.g. if in Figure \ref{crossratioorientation}, we measured $AD$ by the dashed segment, rather than the solid segment). If we traverse in one direction, we have the generalized sine of the length is $\sin(d)$. In the other direction, we include a negative because we flipped the orientation of travel, and we get $-\sin(2\pi-d)=\sin(d)$, so this ratio is well-defined.

It is also preserved by antipodes:

\begin{lemma}\label{AntipodeIndependence}
Let $A$, $B$, $C$, and $D$ be four distinct, pairwise non-antipodal points on an oriented line $\ell$ in $\S^2$. Then, if $A^\ast$ is the antipode of $A$,
\[(ABCD)=(A^\ast BCD).\]
\end{lemma}
\begin{proof}
One can verify that if the distance $AC=d$, then $A^\ast C=\pi\pm d$. Then,
\begin{align*}
    (A^\ast BCD)&=\frac{\sin{A^\ast C}}{\sin{A^\ast D}}\cdot\frac{\sin{BD}}{\sin{BC}} 
    = \frac{\sin{(\pi \pm AC)}}{\sin{(\pi\pm AD)}}\cdot\frac{\sin{BD}}{\sin{BC}} \\
    &= \frac{-\sin{AC}}{-\sin{AD}}\cdot\frac{\sin{BD}}{\sin{BC}} = (ABCD)
\end{align*}
as desired.
\end{proof}

In order to make statements about the cross-ratio, it is helpful to define the cross-ratio of four lines:
\begin{definition}
    Let $\ell_1$, $\ell_2$, $\ell_3$, and $\ell_4$ be four concurrent lines in $\S^2$, $\E^2$, or $\H^2$. Let $\angle\ell_i\ell_j$ be the oriented angle between lines $\ell_i$ and $\ell_j$. Then, the \emph{cross-ratio} of the four lines is
    \[(\ell_1\ell_2\ell_3\ell_4)=\frac{\sin(\angle\ell_1\ell_3)}{\sin(\angle\ell_1\ell_4)}\cdot\frac{\sin(\angle\ell_2\ell_4)}{\sin(\angle\ell_2\ell_3)}.\]
\end{definition}

We now use this definition to prove that a crucial property of the cross-ratio holds in non-Euclidean geometry as well:
\begin{lemma} \label{lemma:fundamentaltheorem}
    Suppose $A$, $B$, $C$, and $D$ are four distinct, pairwise non-antipodal points on a line, and $P$ is a point not on that line in $\S^2$, $\E^2$, or $\H^2$. Then, if we let $\ell_1=AP$, $\ell_2=BP$, $\ell_3=CP$, and $\ell_4=DP$,
    \[(ABCD)=(\ell_1\ell_2\ell_3\ell_4)\]
\end{lemma}

\begin{figure} \begin{center} \includegraphics[width=\textwidth]{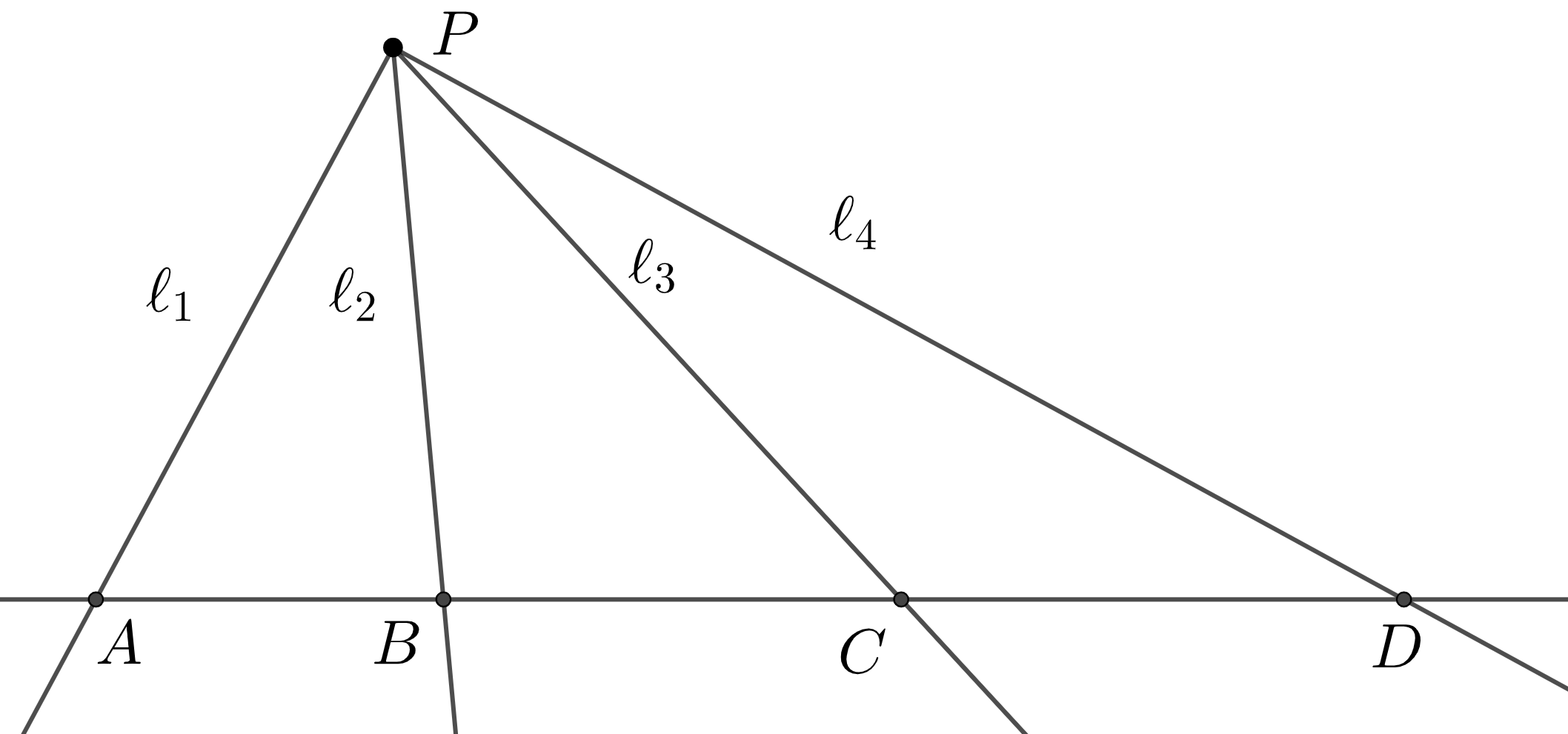} \caption{The cross-ratio $(\ell_1\ell_2\ell_3\ell_4)$ is equal to the cross-ratio $(ABCD)$.}
    \label{fig:fundamentallemma} \end{center} \end{figure} 

A proof by Ibn 'Ir\={a}q from the eleventh century is given in \cite{AncientCrossRatio} for the spherical case, but we provide in $\E^2$, $\S^2$, and $\H^2$ simultaneously for completeness.

\begin{proof}
    Assume without loss of generality that $A$, $B$, $C$, and $D$ lie in that order on a line (see figure \ref{fig:fundamentallemma}) which is oriented so that $AD$ is positive. Further, assume that the lines $AP$, $BP$, $CP$, and $DP$ are oriented so that $AP$, $BP$, $CP$, and $DP$ are all of positive signed length and all angles are oriented counterclockwise. If they lie in a different order or have different orientations, we can make an identical argument to the one that follows. 

    Since our points are distinct and non-antipodal, $(ABCD)\neq 0$. Then, by definition,
    \begin{align*}
        \frac{(ABCD)}{(\ell_1\ell_2\ell_3\ell_4)}&=\frac{\frac{\gsin{AC}}{\gsin{AD}}\cdot\frac{\gsin{BD}}{\gsin{BC}}}{\frac{\sin(\angle\ell_1\ell_3)}{\sin(\angle\ell_1\ell_4)}\cdot\frac{\sin(\angle\ell_2\ell_4)}{\sin(\angle\ell_2\ell_3)}} \\
        &= \frac{\gsin(AC)}{\sin(\angle\ell_1\ell_3)}\cdot \frac{\gsin(BD)}{\sin(\angle\ell_2\ell_4)}\cdot \frac{\sin(\angle\ell_1\ell_4)}{\gsin(AD)} \cdot \frac{\sin(\angle\ell_2\ell_3)}{\gsin(BC)}.
    \end{align*}

    However, as $\ell_1\ell_3=\angle APC$, we can use Equation \eqref{LawOfSines}, the Law of Sines, on the triangle $\triangle APC$ to conclude that
    \[\frac{\gsin(AC)}{\sin(\angle\ell_1\ell_3)} = \frac{\gsin(AC)}{\sin(\angle APC)} = \frac{\gsin(PC)}{\sin(\angle PAC)},\]
    paying close attention to the signs of each length and angle. 
    
    Similarly, as $\ell_2\ell_4=\angle BPD$, we can use the law of sines on $\triangle BPD$ to see that 
    \[\frac{\gsin(BD)}{\sin(\angle\ell_2\ell_4)}=\frac{\gsin(BD)}{\sin(\angle BPD)} = \frac{\gsin(PD)}{\sin(\angle PBD)}.\]
    By identical logic on $\triangle APD$ and $\triangle BPC$, $$\frac{\sin(\angle\ell_1\ell_4)}{\gsin(AD)} = \frac{\sin(\angle PAD)}{\gsin(PD)} \text{ and } \frac{\sin(\angle\ell_2\ell_3)}{\gsin(BC)} =  \frac{\sin(\angle PBC)}{\gsin(PC)}.$$ Thus,
    \begin{align*}
        \frac{(ABCD)}{(\ell_1\ell_2\ell_3\ell_4)} &= \frac{\gsin(AC)}{\sin(\angle\ell_1\ell_3)}\cdot \frac{\gsin(BD)}{\sin(\angle\ell_2\ell_4)}\cdot \frac{\sin(\angle\ell_1\ell_4)}{\gsin(AD)} \cdot \frac{\sin(\angle\ell_2\ell_3)}{\gsin(BC)}\\
        &=\frac{\gsin(PC)}{\sin(\angle PAC)}\cdot\frac{\gsin(PD)}{\sin(\angle PBD)}\cdot\frac{\sin(\angle PAD)}{\gsin(PD)}\cdot\frac{\sin(\angle PBC)}{\gsin(PC)}\\
        &=\frac{\sin(\angle PAD)}{\sin(\angle PAC)}\cdot \frac{\sin(\angle PBC)}{\sin(\angle PBD)} = 1
    \end{align*}
    as $\angle PAD=\angle PAC$ and $\angle PBC=\angle PBD$.
\end{proof}

This result is enough to conclude that perspectivity from one geodesic to another preserves the cross-ratio (see Figure \ref{fig:perspectivity}).

\begin{corollary} Let $A$, $B$, $C$, $D$ lie on a geodesic $\ell$ in $\S^2$, $\E^2$, or $\H^2$, and let $A'$, $B'$, $C'$, $D'$ be their images under a perspectivity from $P$ onto a line $\ell'$. Then,
\[(ABCD)=(A'B'C'D').\] \end{corollary}

\begin{proof}
Observe that $\line{AP}=\line{A'P}$, and similarly for the other points, so
\[(ABCD)=(\line{AP}\line{BP}\line{CP}\line{DP})=(A'B'C'D')\]
as desired. \end{proof}

\begin{figure} 
    \centering
    \includegraphics[scale=0.35]{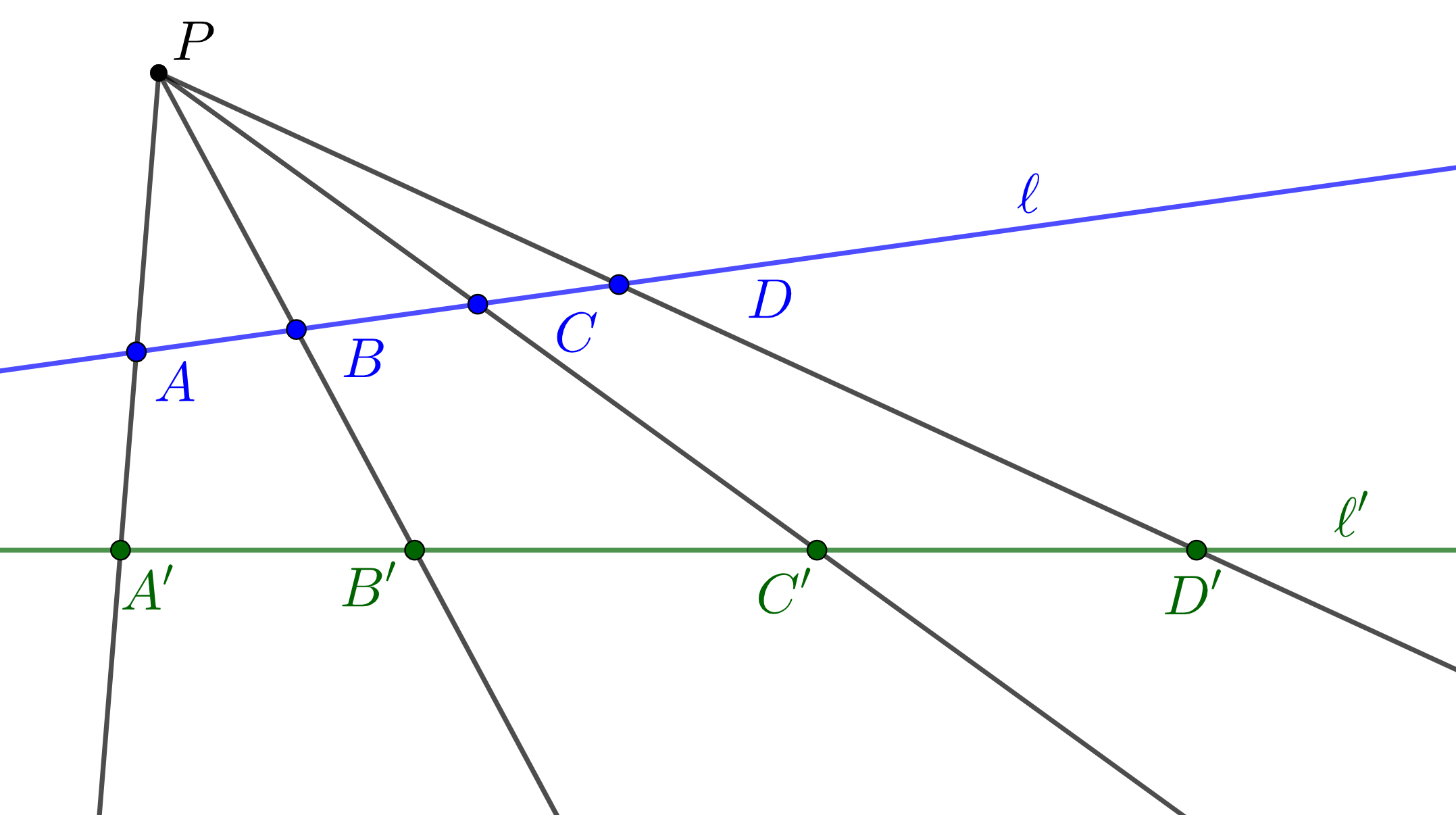}
    \caption{The cross-ratios of the two lines are equal: $(ABCD)=(A'B'C'D')$.}
    \label{fig:perspectivity}
\end{figure}

If this cross-ratio is preserved by perspectivity within each geometry, it is natural if they might also be equivalent under some projection between the different geometries. The answer turns out to be yes.

\begin{theorem} \label{CrossLineCrossPoints}
    Let $A$, $B$, $C$, and $D$ be four distinct, collinear, pairwise non-antipodal points on $\S^2$ or $\H^2$ modeled on the unit sphere or hyperboloid. Let $w$ be a plane in $\R^3$ not passing through the origin so that the images of the four points $A'$, $B'$, $C'$ and $D'$ all exist under central projection. Then,
    \[(ABCD)_{\S^2(\H^2)}=(A'B'C'D')_{\E^2}.\]
\end{theorem}

\begin{proof} First, observe that we only need to prove this for some specially selected plane $w_1$ not containing the origin. Then, the central projection $\pi:\S^2(\H^2)\to w$ for any other plane $w$ is the composition $\pi_2\circ\pi_1$, where $\pi_1:\S^2(\H^2)\to w_1$ and $\pi_2:w_1\to w$ are both central projections from the origin. The latter projection will preserve the cross-ratio of the image points by Theorem \ref{thm:EuclideanInvariance}, so we only need to prove the former projection also preserves the cross-ratio. We will choose different planes $w_1$ for spherical and hyperbolic geometry, so we treat these cases separately.

\emph{Spherical case} Let $A$, $B$, $C$, and $D$ be our four collinear points in $\S^2$ so that no two of them are antipodal. By exchanging points with their antipodes if necessary, we can assume that $A$, $B$, $C$, and $D$ lie in an open hemisphere whose pole does not lie on the geodesic; these exchanges will preserve both the cross-ratio of the four points by Lemma \ref{AntipodeIndependence} and the resulting projection. Then, let $P$ be the pole of that open hemisphere, and let $w_1$ be the tangent plane to $\S^2$ at $P$. (See Figure \ref{fig:SphereEuclidProjection}).

\begin{figure}
    \centering
    \includegraphics[width=\textwidth]{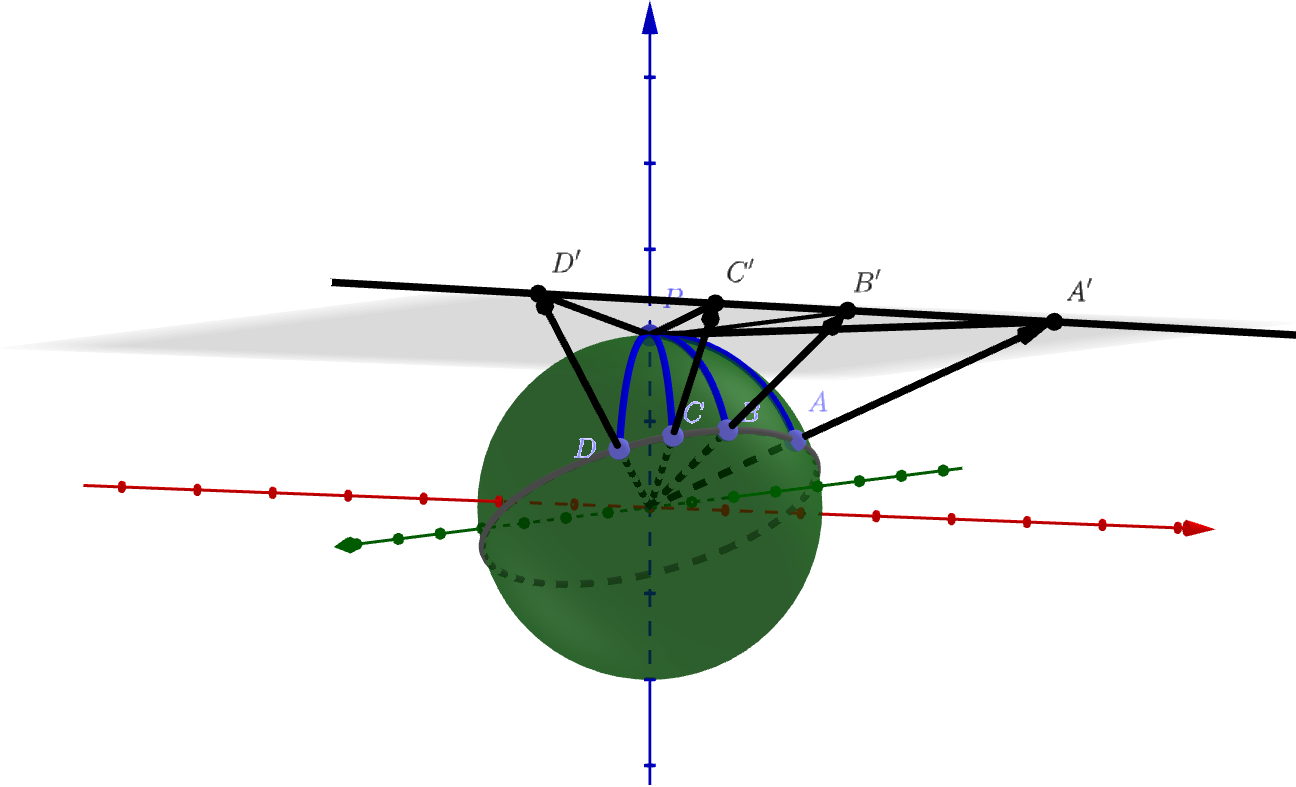}
    \caption{$A,B,C$ and $D$ are projected onto the plane tangent at $P$.}
    \label{fig:SphereEuclidProjection}
\end{figure}

Now, if $A'$, $B'$, $C'$, and $D'$ are the projections onto $w_1$, observe that by definition, $\angle_{\S^2} APC=\angle_{\E^2} A'PC'$ as $w_1$ is the tangent plane at $P$. Similarly, $\angle_{\S^2} BPC=\angle_{\E^2} B'PC'$ and so on. Consequently, if $\ell_1=\line{AP}$, $\ell_2=\line{BP}$, $\ell_3=\line{CP}$ $\ell_4=\line{DP}$, and $\ell_1'=\line{A'P}$, $\ell'_2=\line{B'P}$, $\ell'_3=\line{C'P}$ and $\ell'_4=\line{D'P}$, then $\angle\ell_i\ell_j=\angle\ell_i'\ell_j'$. Thus, $(\ell_1\ell_2\ell_3\ell_4)=(\ell_1'\ell_2'\ell_3'\ell_4')$, so we can use Lemma \ref{lemma:fundamentaltheorem} to conclude that
\[(ABCD)_{\S^2}=(\ell_1\ell_2\ell_3\ell_4)=(\ell_1'\ell_2'\ell_3'\ell_4')=(A'B'C'D')_{\E^2}.\]

\emph{Hyperbolic Case} Let $w_1$ be the plane $z=1,$ which is the tangent plane to the hyperboloid at $P=(0,0,1)$. First, suppose that the line $\line{ABCD}$ does not pass through $P$. The plane $w_1$, while embedded in Minkowski $3$-space, has a Euclidean metric (as $dz=0$, so the metric is $ds^2=dx^2+dy^2$), so by identical logic as the $\S^2$ case, $\angle_{\H^2} APC=\angle_{\E^2} A'PC'$. The same logic holds for the other angles, so by Lemma \ref{lemma:fundamentaltheorem},
\[(ABCD)_{\H^2}=(\line{AP}\ \line{BP}\ \line{CP}\ \line{DP}) = (\line{A'P}\ \line{B'P}\ \line{C'P}\ \line{D'P}) = (A'B'C'D')_{\E^2}.\]

Finally, suppose that $P=(0,0,1)$ lies on $\line{ABCD}$. Consider smooth maps $A_x,B_x,C_x,D_x:(-\epsilon,\epsilon)\to\H^2$ that approach $A$, $B$, $C$, and $D$ respectively as $x\to 0$, so that $A_x,B_x,C_x$ and $D_x$ lie on a geodesic not passing through $P$ for $x\neq 0$. Since the cross-ratio is a continuous function, $(A_xB_xC_xD_x)\to (ABCD)$ as $x\to 0$. Similarly, $(A'_xB'_xC'_xD'_x)\to (A'B'C'D')$ as $x\to 0$. However, for $x\neq 0$, we can use the preceding discussion to conclude that $(A_xB_xC_xD_x)=(A_x'B_x'C_x'D_x')$. Then, by a continuity argument, \[(ABCD)=(A'B'C'D').\] 
The theorem holds.
\end{proof}

This theorem admits two immediate corollaries:

\begin{corollary}
    If $\ell_1$, $\ell_2$, $\ell_3$, and $\ell_4$ are four concurrent lines in $\H^2$ or $\S^2$ and $\ell_i'$ is the image of $\ell_i$ under central projection from the origin to a plane $w$
    \[(\ell_1\ell_2\ell_3\ell_4)=(\ell_1'\ell_2'\ell_3'\ell_4')_{\E^2}.\]
\end{corollary}

While we used a weakened version of this corollary to prove Theorem \ref{CrossLineCrossPoints}, this is not immediately be obvious, especially in hyperbolic geometry. The left hand side features angles that are measured in the Minkowski metric, while angles on the right-hand side are measured in the Euclidean metric! This suggests that the cross-ratio of four angles in Minkowski geometry is equal to the cross-ratio of the same four angles in Euclidean geometry, which is not obvious.

\begin{corollary}
    Let $A$, $B$, $C$, $D$ be any four points on a line in the Beltrami Klein Model. Then, if we let ${AB}$ be the signed euclidean distance between two points, and $AB$ be the signed hyperbolic distance,
    \[\frac{{AC}}{{AD}}\cdot\frac{{BD}}{{BC}}=\frac{\sinh{AC}}{\sinh{AD}}\cdot\frac{\sinh{BD}}{\sinh{BC}}.\]
\end{corollary}
\begin{proof}
    As noted earlier, central projection of the hyperboloid model onto the plane $z=1$ creates the Beltrami Klein model. The hyperbolic distance corresponds to the distance of the pre-image of the four points, while the Euclidean distances correspond to the image under the projection, so we can apply our theorem.
\end{proof}

Beyond the immediate corollaries noted above, this section illustrates how to use Theorem \ref{CrossLineCrossPoints} as a tool to prove a large class of theorems directly. Because central projection preserves incidence, geodesics, conics, and a notion of the cross-ratio, any theorem that only depends on those properties is, with minor modifications, still true in spherical and hyperbolic geometry. 

For instance, we can generalize Chasles's theorem for conics: 
\begin{figure}[H]
\begin{center}
    \includegraphics[scale=0.4]{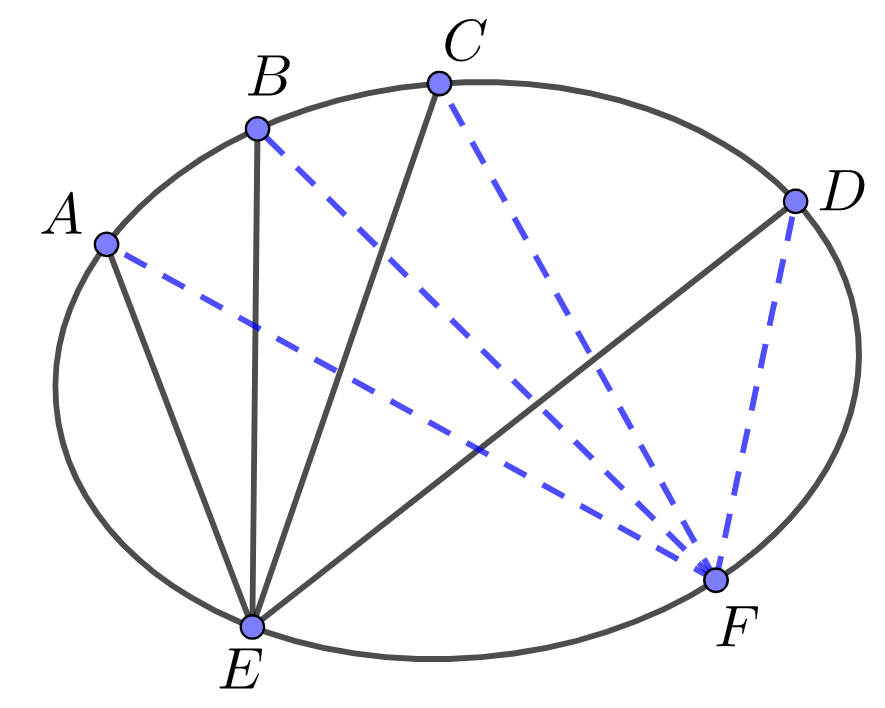}
    \caption{The cross-ratio of the four solid lines from $E$ is the same as the cross-ratio of the four dashed lines from $F$}
    \label{fig:chasles}
\end{center}
\end{figure}
\begin{theorem}
    Let $A$, $B$, $C$, $D$, $E$, and $F$ be any six points on a non-degenerate conic in $\E^2$ (see figure \ref{fig:chasles}). Then,
    \[(\line{AE}\line{BE}\line{CE}\line{DE})=(\line{AF}\line{BF}\line{CF}\line{DF}).\]
\end{theorem}

As this only depends on conics and the cross-ratio, we can immediately generalize this to hyperbolic and spherical geometry:

\begin{corollary}
    Theorem 9 holds in $\S^2$ and $\H^2$.
\end{corollary}

Moreover, we can also use this to prove theorems whose proofs (but not their statements) depend only on these tools. For instance, we can extend the beautiful Butterfly theorem into spherical and hyperbolic geometry by proving it through cross-ratios:

\begin{theorem}[Butterfly Theorem] Let $PQ$, $AB$, and $CD$ be chords on a conic in $\S^2$, $\E^2$, or $\H^2$ so that $AB$ and $CD$ pass through the midpoint $M$ of $PQ$. Then, if $AD$ and $BC$ intersect $PQ$ at $X$ and $Y$ respectively, $M$ is the midpoint of $XY$. \end{theorem}

\begin{center}
\includegraphics[scale=0.5]{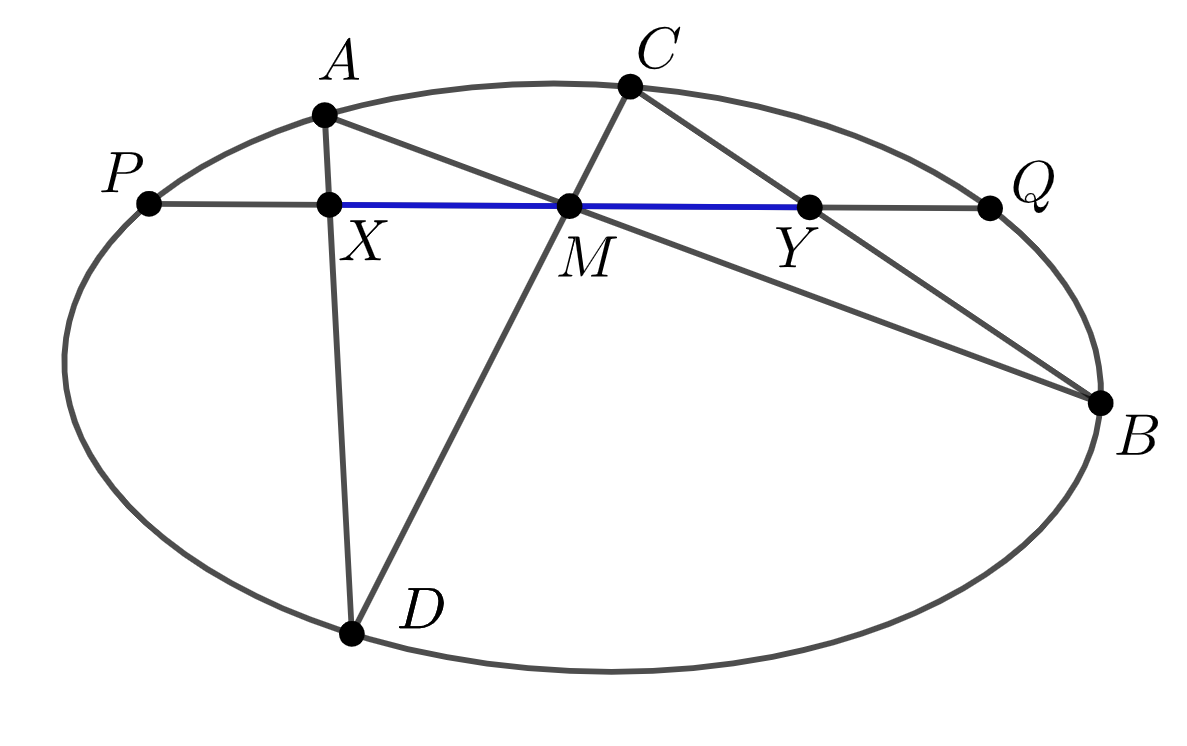}
\end{center}

\begin{proof} We adapt the proof given in \cite{CrossRatioReference} for Euclidean geometry. Observe that 
\begin{align*}
    (PQMX)&=(\line{AP}\line{AQ}\line{AB}\line{AD}) \tag{\text{Lemma }\ref{lemma:fundamentaltheorem}} \\
    &=(\line{CP}\line{CQ}\line{CB}\line{CD}) \tag{\text{Chasles's Theorem}} \\
    &=(PQYM) \tag{\text{Lemma }\ref{lemma:fundamentaltheorem}}.
\end{align*}
Then, $(PQMX)=\frac{\gsin{PM}}{\gsin{PX}}\cdot\frac{\gsin{QX}}{\gsin{QM}}=\frac{\gsin{QX}}{\gsin{PX}}$ as $PM=QM$. Similarly, $(PQYM)=\frac{\gsin{PY}}{\gsin{QY}}$. But, this implies that $XM=MY$ as desired.
\end{proof}

There are many generalizations of the Butterfly Theorem, and we could use similar ideas to painlessly generalize them to $\S^2$ and $\H^2$ as well.

\section{From Cross-ratios to Carnot's Theorem}

The goal of this section is to illustrate the power of this extension of the cross-ratio by using to prove a generalization of Carnot's Theorem for conics which was introduced in the introduction. 
Another tool we will need is Menelaus's Theorem:

\begin{theorem} Let $\triangle ABC$ be a triangle in $\S^2$, $\E^2$, or $\H^2$ (see figure \ref{fig:MenelausandCarnotRewrite}). Then, if points $A_\ell$, $B_\ell$, and $C_\ell$ 
lie on $\line{BC}$, $\line{AC}$, and $\line{AB}$ respectively, $A_\ell$, $B_\ell$, and $C_\ell$ are collinear if and only if
\[\frac{\gsin(AC_\ell)}{\gsin(C_\ell B)}\cdot\frac{\gsin(BA_\ell)}{\gsin(A_\ell C)}\cdot\frac{\gsin(CB_\ell)}{\gsin(B_\ell A)}=-1,\]
where $\gsin(P_1P_2)$ represents the generalized sign of the oriented distances between $P_1$ and $P_2$. \end{theorem}

One can find a lovely proof of this result in \cite{VolumePrinciple}. Observe that the above ratio is preserved if we 
replace any point in the arrangement with its antipode. For instance, if we replace $A_1$ with its antipode $A_1^\ast$,
then $BA_1^\ast=\pi\pm BA_1$, and $A_1C^\ast=\pi\pm A_1C$. Consequently, $\gsin(BA_1^\ast)=-\gsin(BA_1)$ and 
$\gsin(A_1^\ast C)=-\gsin(A_1C)$, and the negative signs cancel each other out.

\begin{figure}
    \centering
    \includegraphics[width=\textwidth]{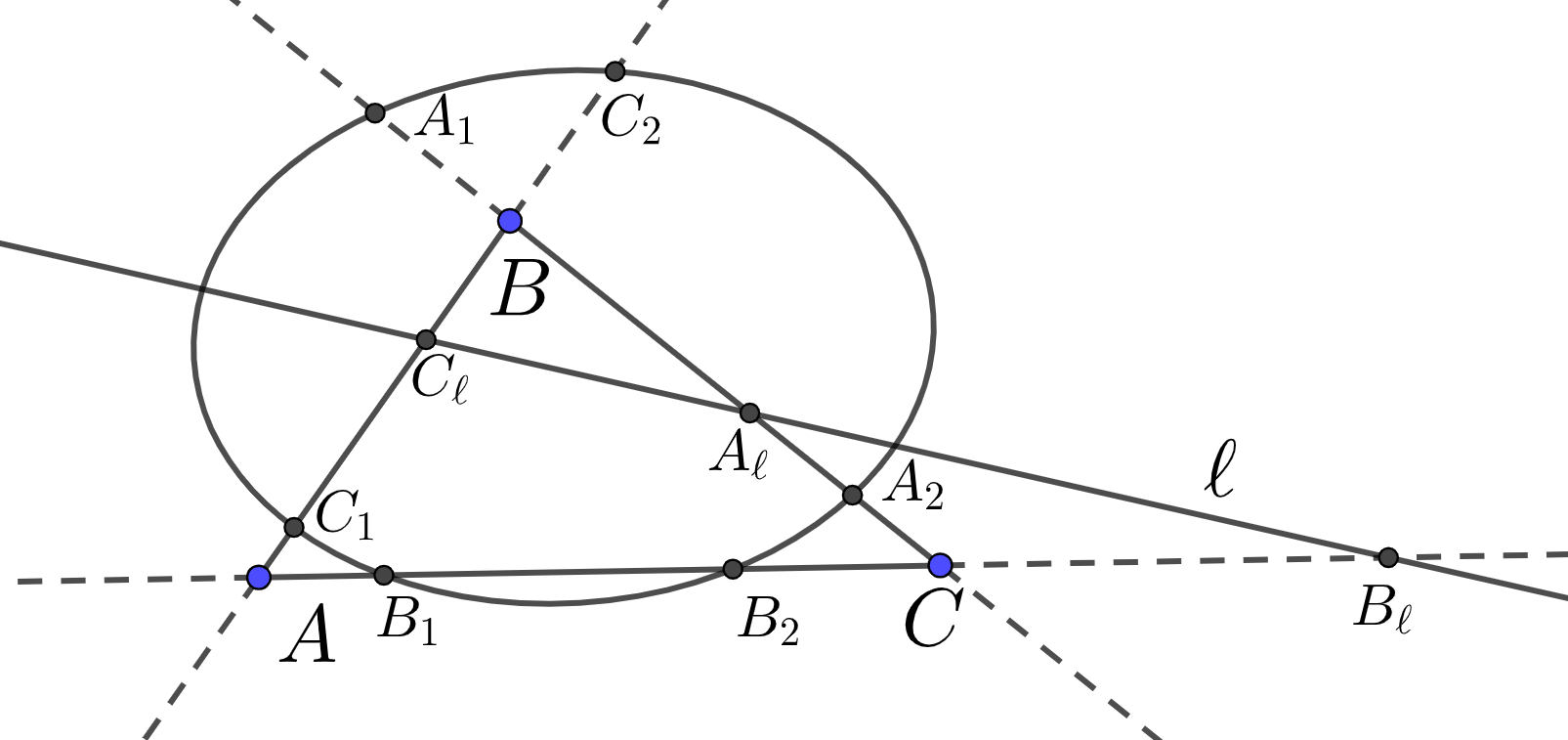}
    \caption{Menelaus's Theorem and Lemma \ref{lemma:carnotrewrite}}
    \label{fig:MenelausandCarnotRewrite}
\end{figure}

With this tool, we are ready to state our generalization of Carnot's Theorem:

\begin{theorem} \label{GeneralCarnotThm} Let $ABC$ be an arbitrary triangle in $\S^2$, $\R^2$, or $\H^2$. Let $A_1$, $A_2$; $B_1$, $B_2$, and $C_1$; $C_2$ be points on the sides of oriented lines $\line{BC}$, $\line{AC}$, and $\line{AB}$ respectively. Then, $A_1$, $A_2$; $B_1$, $B_2$; and $C_1$, $C_2$ lie on a conic if and only if
\begin{equation} \label{GeneralCarnot}
\frac{\gsin(AC_1)}{\gsin(C_1B)}\cdot\frac{\gsin(AC_2)}{\gsin(C_2B)}\cdot\frac{\gsin(BA_1)}{\gsin(A_1C)}\cdot\frac{\gsin(BA_2)}{\gsin(A_2C)}\cdot\frac{\gsin(CB_1)}{\gsin(B_1A)}\cdot\frac{\gsin(CB_2)}{\gsin(B_2A)} = 1.
\end{equation}
\end{theorem}

Before we prove the theorem, we first prove a lemma rewriting the equation in terms of cross-ratios:

\begin{lemma}\label{lemma:carnotrewrite}
    Suppose $\ell$ is a line not passing through any of $A$, $B$, or $C$. Then, if $\ell$ intersects the sides of the triangle at $A_\ell$, $B_\ell$, and $C_\ell$, the Equation \eqref{GeneralCarnot} is equivalent to
    \[(ABC_\ell C_1)(ABC_\ell C_2)(BCA_\ell A_1)(BCA_\ell A_2)(CAB_\ell B_1)(CAB_\ell B_2)=1.\]
\end{lemma}

\begin{proof}
    This is a straightforward computation (refer to Figure \ref{fig:MenelausandCarnotRewrite}): because $A_\ell$, $B_\ell$, and $C_\ell$ lie on a line, we know that 
    \begin{equation}\label{Menelaus} \frac{\gsin(AC_\ell)}{\gsin(C_\ell B)}\cdot\frac{\gsin(BA_\ell)}{\gsin(A_\ell C)}\cdot\frac{\gsin(CB_\ell)} {\gsin(B_\ell A)}=-1. \end{equation}
    Then, we can multiply \eqref{GeneralCarnot} by the square of \eqref{Menelaus} to get
    \begin{align*}
    1(-1)^2&=\frac{\gsin(AC_1)}{\gsin(C_1B)}\cdot\frac{\gsin(AC_2)}{\gsin(C_2B)}\cdot\frac{\gsin(BA_1)}{\gsin(A_1C)}\cdot\frac{\gsin(BA_2)}{\gsin(A_2C)}\\  &\qquad\cdot\frac{\gsin(CB_1)}{\gsin(B_1A)}\cdot\frac{\gsin(CB_2)}{\gsin(B_2A)} \left(\frac{\gsin(AC_\ell)}{\gsin(C_\ell B)}\cdot\frac{\gsin(BA_\ell)}{\gsin(A_\ell C)}\cdot\frac{\gsin(CB_\ell)} {\gsin(B_\ell A)}\right)^2  \\
    1&=\left(\frac{\gsin(AC_\ell)}{\gsin(AC_1)}\cdot\frac{\gsin(BC_1)}{\gsin(BC_\ell)}\right)\cdot\left(\frac{\gsin(AC_\ell)}{\gsin(AC_2)}\cdot\frac{\gsin(BC_2)}{\gsin(BC_\ell)}\right) \\
    &\qquad \cdot\left(\frac{\gsin(BA_\ell)}{\gsin(BA_1)}\cdot\frac{\gsin(CA_1)}{\gsin(CA_\ell)}\right)  \left(\frac{\gsin(BA_\ell)}{\gsin(BA_2)}\cdot\frac{\gsin(CA_2)}{\gsin(CA_\ell)}\right)\\
    &\qquad\,\cdot \left(\frac{\gsin(CB_\ell)}{\gsin(CB_1)}\cdot\frac{\gsin(AB_1)}{\gsin(AB_\ell)}\right) \cdot \left(\frac{\gsin(CB_\ell)}{\gsin(CB_2)}\cdot\frac{\gsin(AB_2)}{\gsin(AB_\ell)}\right) \\
    1&= (ABC_\ell C_1)(ABC_\ell C_2)(BCA_\ell A_1)(BCA_\ell A_2)(CAB_\ell B_1)(CAB_\ell B_2).
    \end{align*}
    As the steps are reversible, the two equations are equivalent.
\end{proof}

With this lemma, we can now prove Theorem \ref{GeneralCarnotThm}.

\begin{proof}[Proof of Theorem] Our strategy is as follows. First, note the theorem is true in $\E^2$ (\cite{CarnotApplications} gives a proof). In $\S^2$ or $\H^2$, we will project the arrangement into a plane $w$; then we use the information gleaned from the Euclidean Carnot's theorem to prove the theorem.

Suppose we have $\triangle ABC$ and $A_1,A_2,B_1,B_2,C_1,$ and $C_2$ as described in the theorem statement in $\S^2$ or $\H^2$. In $\S^2$, further suppose that all of these points lie in an open hemisphere. If not, take antipodes until they do; antipodal points lie on the same conics, and none of the equations are changed by taking the antipode.

Now, given such points, take some line $\ell$ that does not pass through any of the points, but intersects $\line{BC}$, $\line{AC}$ and $\line{AB}$ at $A_\ell$, $B_\ell$, $C_\ell$ (in $\S^2$, choose the line and intersection points so that they lie in our open hemisphere). Then, we can take this entire arrangement and centrally project it onto the tangent plane $w$ whose point of tangency is the center of the open hemisphere in $\S^2$, or onto the plane $z=1$ in $\H^2$.

Observe then that if $A'$, $B'$, $C'$, $A'_1, A'_2, B'_1, B'_2, C'_1$ and $C'_2$ are the images of the above points, the six points $A'_1, A'_2, B'_1, B'_2, C'_1$ and $C'_2$ lie on a conic if and only if $A_1,A_2,B_1,B_2,C_1,$ and $C_2$ lie on a conic. Since we are now in $\E^2$, we know by Carnot's Theorem that the Equation \eqref{EuclidCarnot} holds if and only if the six points $A'_1, A'_2, B'_1, B'_2, C'_1$ and $C'_2$ lie on a conic.

By our previous lemma, this is true if and only if
\begin{equation*} \label{EuclidenCarnotCross} (A'B'C_\ell' C_1')(A'B'C_\ell' C_2')(B'C'A_\ell' A_1')(B'C'A_\ell' A_2')(C'A'B_\ell' B_1')(C'A'B_\ell' B_2')=1.\end{equation*}

Since all of these points are the images under a central projection, this equation only holds if the corresponding equation in $\S^2$ (or $\H^2$) holds because central projection preserves the cross-ratio as proved in Theorem \ref{CrossLineCrossPoints}. Therefore, the six points $A_1,A_2,B_1,B_2,C_1,$ and $C_2$ lie on a conic if and only if 
\[ (ABC_\ell C_1)(ABC_\ell C_2)(BCA_\ell A_1)(BCA_\ell A_2)(CAB_\ell B_1)(CAB_\ell B_2)=1. \]

By the lemma, this is true if and only if \eqref{GeneralCarnot} holds. Therefore, $A_1,A_2,B_1,B_2,C_1,$ and $C_2$ lie on a conic if and only if the cyclic product relation \eqref{GeneralCarnot} holds, and the theorem follows.
\end{proof}

\section{Higher Degree Curves}
In \cite{CarnotApplications}, the author notes that Carnot's theorem is a generalization of Menelaus's theorem to a curve of degree $2$, and then shows that there is an extension to higher degree curves. We can cleanly generalize that extension to spherical and hyperbolic geometry. 

First, we need to construct a notion of higher degree curves in $\S^2$ or $\H^2$. We note that our definition of both a geodesic (a degree 1 curve) and a conic (a degree 2 curve) consisted of taking the corresponding object on a Euclidean plane and then projecting it onto the sphere (or the hyperboloid). We can do the same thing with a curve of degree $n$:

\begin{definition}
    A curve in $\S^2$ or $\H^2$ has degree $n$ if and only if its central projection onto a plane $w$ is a degree $n$ curve on the plane $w$.
\end{definition}

Then, from \cite{MathEncyclopedia,CarnotApplications} we know that if $\triangle ABC$ has $n$-tuples of points $A_1,\ldots,A_n$, $B_1,\ldots,B_n$, and $C_1,\ldots,C_n$ on the sides $\line{BC}$, $\line{AC}$, and $\line{AB}$ respectively, these $3n$ points lie on a degree $n$ curve if and only if 
\begin{equation}\label{CarnotGeneral}\prod_{k=1}^n \frac{AB_k}{B_kC}\cdot\frac{BC_k}{C_kA}\cdot\frac{CA_k}{A_kB}=(-1)^n.\end{equation}

Then, if we mimic Lemma \ref{lemma:carnotrewrite}, we can introduce a line $\ell$ that does not intersect with any $A_i$, $B_i$, $C_i$. Using Menelaus's Theorem, we can rewrite Equation \eqref{CarnotGeneral} as 
\[\prod\limits_{k=1}^n (ABC_\ell C_k)(BCA_\ell A_k)(CAB_\ell B_k)=(-1)^n.\]

Then, using identical logic to the proof of Carnot's Theorem in the previous section, we can generalize Carnot's theorem as follows:

\begin{theorem}
    Let $\triangle ABC$ be a triangle in $\S^2$, $\E^2$, or $\H^2$. Then, if $A_1,\ldots,A_n$, $B_1\ldots,B_n$, and $C_1,\ldots C_n$ are $n$-tuples of points lying along $\line{BC}$, $\line{AC}$, and $\line{AB}$ respectively, they lie on a curve of degree $n$ if and only if 
    \[\prod_{k=1}^n \frac{\gsin(AB_k)}{\gsin(B_kC)}\cdot\frac{\gsin(BC_k)}{\gsin(C_kA)}\cdot\frac{\gsin(CA_k)}{\gsin(A_kB)}=(-1)^n.\]
\end{theorem}

\noindent \textbf{Acknowledgements} This paper came out of work done at the 2022 Grand Val-
ley State University REU and was supported by the NSA [Grant No: H98230-22-
1-0023]. We would also like to thank our advisor, Professor William Dickinson,
for his feedback and support during (and after!) the summer, as well as Derrick
Wu for his comments.

\begin{abstract}
    When considering geometry, one might think of working with lines and circles on a flat plane as in Euclidean geometry. However, doing geometry in other spaces is possible, as the existence of spherical and hyperbolic geometry demonstrates. Despite the differences between these three geometries, striking connections appear among the three. In this paper, we illuminate one such connection by generalizing the cross-ratio, a powerful invariant associating a number to four points on a line, into non-Euclidean geometry.  Along the way, we see how projections between these geometries can allow us to directly export results from one geometry into the others. The paper culminates by generalizing Carnot’s Theorem for Conics -- a classical result relating when six points on a triangle lie on a conic -- into spherical and hyperbolic geometry. These same techniques are then applied to Carnot’s Theorem for higher degree curves.
\end{abstract}

\end{document}